\documentclass[a4paper,12pt,sumlimits,intlimits]{amsart}

\usepackage[T1]{fontenc}
\usepackage[latin1]{inputenc}

\usepackage{amssymb,amsthm,upref}
\usepackage[all]{onlyamsmath}
 \usepackage{mathtools}
 \mathtoolsset{centercolon}

\usepackage[twoside]{geometry}
\usepackage{url}
\usepackage{xcolor}
\usepackage{enumerate}
\usepackage{braket}
\usepackage{xspace}
\usepackage{leading}

\usepackage{hyperref}
\hypersetup{pdfborder= 0 0 0}
\hypersetup{linkcolor=magenta}
\hypersetup{colorlinks=true}
\hypersetup{allcolors=blue!65!gray}
\hypersetup{
 pdfinfo={Title={A kinematic formula for the total absolute curvature},
          Subject={Integral geometry},
          Author={Ludwig Br\"ocker},
          Keywords={kinematic formula, curvature measures, piecewise linear
            geometry}
 }
}

\newcommand{\Gauss}{Gau\ss\xspace}
\newcommand{\Bonnet}{Bonnet\xspace}
\newcommand{\Chern}{Chern\xspace}
\newcommand{\Lashof}{Lashof\xspace}

\newcommand{\ccap}{\,\cap\,}
\renewcommand{\leq}{\leqslant}
\renewcommand{\geq}{\geqslant}

\newcommand{\R}{{\mathbb R}}
\newcommand{\T}{\mathcal{T}}

\newcommand{\PL}{\mathrm{PL}}
\newcommand{\vol}{\mathrm{vol}}
\newcommand{\sign}{\mathrm{sign}}
\newcommand{\backIII}{\!\!\!}
\newcommand{\backIV}{\!\!\!\!}
\newcommand{\backV}{\!\!\!\!\!}
\renewcommand{\atop}[2]{\genfrac{}{}{0pt}{}{#1}{#2}}

\DeclareMathOperator{\Aff}{Aff}
\DeclareMathOperator{\Fr}{Fr}
\DeclareMathOperator{\Graff}{Graff}
\DeclareMathOperator{\Int}{Int}
\DeclareMathOperator{\ret}{ret}

\newtheorem{theorem}{Theorem}[section]

\newtheorem{proposition}[theorem]{Proposition}
\newtheorem{corollary}[theorem]{Corollary}
  \newtheorem{remark}[theorem]{Remark}
  \newtheorem*{remark*}{Remark}
  \newtheorem{example}[theorem]{Example}
  \newtheorem*{example*}{Example}

\theoremstyle{definition}
  \newtheorem{definition}[theorem]{Definition}
  \newtheorem{notation}[theorem]{Notation}
  \newtheorem*{notation*}{Notation}

\leading{15pt}
\title{A kinematic formula for \\ the total absolute curvature}
\author{Ludwig Br\"ocker}
\address{Ludwig Br\"ocker  \\ Universit\"at M\"unster \\
Mathematisches Institut \\ Einsteinstr.\ 62,
 D-48149 M\"unster 
}
\email{broe@wwu.de}
\keywords{Kinematic formula, curvature measures, piecewise linear geometry}
\subjclass[2010]{53C65, 51M20}
\date{April 10,\, 2018}

\begin{document}

\begin{abstract}
  Let $A$ and $B$ be compact PL-subspaces of some euclidean space $E^n$. We
  show that for these a kinematic formula for the total absolute curvature
  holds in analogy to the classical one.
\end{abstract}

\maketitle

\pagestyle{myheadings} \markboth{LUDWIG BR\"OCKER}{A KINEMATIC FORMULA FOR
  THE TOTAL ABSOLUTE CURVATURE}

\section*{Introduction}
\label{sec:intro}

Let $E^n$ be an euclidean space and $E^n \supset A, B$ be compact tame
sets. The notion ``tame'' will be explained from case to case. The classical
kinematic formula says:
\[ 
  \int\limits_G\chi(A \cap gB)\, dG = \sum_{k=0}^n c_k V_k(A)V_{n-k}(B)\;.
\]
Here $\chi$ is the Euler characteristic, the $c_k$ are universal constants:
\[
  c_k =c(n,k):=\binom{n}{k}^{-1}\frac{\omega_k\,\omega_{n-k}}{\omega_n}\;,
\] 
where $\omega_i$ is the volume of the $i$-dimensional unit ball and $G = O(n)
\ltimes \R^n$ is the group of all euclidean motions of $\R^n$, endowed with
the product of Haar measure and Lebesgue measure. The $V_k$, $k=0,\dots,n$ are
functionals for tame sets, which are known under different names:
Minkowski--functional, cross sectional measure (Querma\ss), generalized
volumes, Lipschitz--Killing invariant, \dots One of the possible definitions
is:
\[
  V_k(A) = 
  \smashoperator{\int\limits_{\Graff(n,n-k)}}\chi(A\cap E)\,dE\,,
\]
where $\Graff(n,n-k)$ is the affine Grassmannian of $(n-k)$-planes in $E^n$,
provided with a $G$-invariant measure. In Section~\ref{sec:PL-spaces} for
$\PL$-spaces (piecewise linear spaces) we will consider a more intrinsic
definition for the $V_k$ and similarly for absolute curvature measures
(compare Notation~\ref{not:2.3} and Corollary~\ref{cor:3.8}).

Now, concerning the notion ``tame'' it is more reasonable to consider
\emph{tame classes} of sets. So, tame sets are members of a tame class.

The kinematic formula is known for the following tame classes:
\begin{itemize}
\item[$\circ$] Convex sets (Blaschke, Hadwinger and many others~\cite{K-R}),
\item[$\circ$] Manifolds with and without boundary (Chern~\cite{C},
  Santalo~\cite{S}),
\item[$\circ$] $\PL$-sets (Wintgen, Cheeger--M\"uller--Schrader~\cite{C-M-S}),
\item[$\circ$] Sets with positive reach (Rataj, Z\"ahle~\cite{R-Z},~\cite{Z}),
\item[$\circ$] Definable sets with respect to an o-minimal
  system~\cite{vdD}, in particular subanalytic sets (Fu~\cite{Fu2},
  Br\"ocker--Kuppe~\cite{Br-K}).
\end{itemize}
For more details and generalizations see also~\cite{Sch-W},~\cite{H-Sch}.

Let $\dim(A)=m$. One has $V_0=\chi(A)$, $V_m(A)= \vol_m(A)$,
$V_{m-1}(A)=\frac{1}{2}\vol_{m-1}(\partial A)$ if $A$ is a manifold with
boundary, $V_k(A)=0$ for $k>m$.

We are going to show a kinematic formula for the absolute curvature, to be
defined below, at least for $\PL$-sets. On the way, we find a new proof for
the usual kinematic formula in the $\PL$-case, which does not use
approximation by manifolds and reduction to Chern's result as in~\cite{C-M-S}.

However, kinematic formulas for the total absolute curvature are not known
for other settings not even for manifolds except for special situations in
dimension $2$, see~\cite{G-R-S-T} and for linear kinematic
formulas~\cite{Ba},~\cite{R-Z} (compare also Corollary~\ref{cor:3.8}). 
The reason for this is, that the total absolute curvature is very sensitive
concerning approximation by triangulations.

I thank Christian Gorzel for careful reading  and typing.

\section{Curvature measures}
\label{sec:curvmaesr}

Let $E^n\supset X$ be a compact manifold, possibly with boundary.
Let 
\[
N(X):=\Set{ (x,a) | x \in X,\, a\in S^{n-1},\, a\perp T_x(X)}
\]
be the unit normal bundle. We denote by
\[
  \gamma: N(X) \to S^{n-1}\;, \quad (x,a)\mapsto a
\]
the \Gauss map.

The absolute curvature $\tau(X)$ is defined by
\[
 \tau(X) = \frac{1}{|S^{n-1}|} \int\limits_{S^{n-1}}\!\!\negthickspace
 \#\;\!\gamma^{-1}(a)\,da\;.
\]
Similarly, one has the \Gauss  curvature, for which we choose an orientation
on $N(X)$. Then
\[
\sigma(X) = \frac{1}{|S^{n-1}|}\int\limits_{S^{n-1}} \sum_{(x,a)\in N(X)}
 \negthickspace \negthickspace \sign(\gamma(x,a)) \,da\;.
\]
One has
\begin{alignat*}{2}
  \sigma(X) & = \chi(X) \qquad & \qquad &\text{\Gauss--\Bonnet}, \\
  \tau(X) & \geq \sum_{k=0}^n b_k(X) \qquad & \qquad & \text{\Chern--\Lashof}\,.
\end{alignat*}
Here $b_0, b_1, b_2,\dots, b_n$ are the Betti numbers.  On the way, we will
prove these formulas in a more general setting. Note that the quantities
under the integrals are finite up to a set of measure $0$ in $S^{n-1}$. 

One may localize $\sigma$ and $\tau$, thus getting curvature measures. Also
this we will do in a more general setting.  
Let $X\subset E^n$, where $X$ is compact and belongs to a tame class. Let
$x\in X$. For $r> 0$ the homeomorphy type of $B(x,r)\cap X$ does
not depend on $r$ if $r$ is sufficiently small.

\begin{definition}
 \label{def:1.1}
  Let $x \in X$, $r>0$ and $0 < \delta \ll r$. Also let $a\in
  S^{n-1}$. Then
  \begin{align*}
    C(X,x,a) & = \Set{ y \in X \cap B(x,r) | -\delta \leq \langle y-x,a
      \rangle \leq \delta}
\intertext{is called the \emph{cone of $X$ at $x$ in direction $a$}, and}
    L(X,x,a) & = \Set{ y \in X \cap B(x,r) | \langle y-x,a \rangle 
                 = \delta}
  \end{align*}
is called the \emph{link of $X$ at $x$ in direction $a$}.
\end{definition}

Again, up to homeomorphy, the pair of spaces $\bigl(C(X,x,a)/ L(X,x,a) \bigr)$
is independent of the choices of $r$ and $\delta$.

\begin{definition} \label{def:1.2}
\mbox{}\par
  \begin{enumerate}[{\upshape a)}]
  \item \quad
  $\sigma(X,x,a):=1-\chi(L(X,x,a))$ is called the \emph{index of $X$ at $x$
  in direction $a$}. \\
\item \quad 
  $\tau(X,x,a):=|b(L(X,x,a))-1|$ is called the \emph{absolute index of $X$
  at $x$ in direction $a$}.
 Here $b$ is the sum of all Betti numbers, that is:
\[
  b(L(X,x,a))=b_0(L(X,x,a))+ \dots + b_n(L(X,x,a))\;.
\]
\end{enumerate}
\end{definition}

\begin{remark}[compare~\cite{Kue}]
\label{rem:1.3}
  \mbox{}\par
  \begin{enumerate}[{\upshape a)}]
  \item \quad $\sigma(X,x,a)=\displaystyle{\sum\limits_{k=0}^n} (-1)^k
    b_k(C(X,x,a)/L(X,x,a))$\;. \\
  \item \quad $\tau(X,x,a)=\displaystyle{\sum\limits_{k=0}^n}
    b_k(C(X,x,a)/L(X,x,a))$\;.
  \end{enumerate}
This can be seen by regarding the long exact homology sequence together with
the fact that $C(X,x,a)$ is contractible.
\end{remark}

Let $a\in S^{n-1}$ such that $\varphi(y):= -\langle a, y \rangle$ is a Morse
function~\cite[Chap.\,2.1]{G-M}. Suppose that $x$ is a critical point for
$\varphi$ and that there is no other critical point $x'$ with
$\varphi(x')=\varphi(x)=:\alpha$. Then $\varphi^{-1}(\alpha+\delta)$ is
homeomorphic to the space which one gets by attaching $C(X,x,a)$ at
$\varphi^{-1}(\alpha-\delta)$ along $L(X,x,a)$~\cite[3.5.4]{G-M}.

If there are several critical points
, where $\varphi$ takes the
same value, then the result is similar. One has just to attach the different
cones simultaneously. Anyway, there are only finitely many critical points
for $\varphi$. Thus we get:

\begin{proposition}
\label{prop:1.4}
  In the situation above:
  \begin{enumerate}[{\upshape a)}]
  \item \quad $\chi(X)= \displaystyle{%
      \sum\limits_{\mathclap{ x\text{ critical for }\varphi}} }\, \sigma(X,x,a)$\,,\\
  \item \quad $\displaystyle{\sum\limits_{k=0}^n b_k(X)\leq
      \sum\limits_{\mathclap{ x\text{ critical for }\varphi}} \tau(X,x,a)}$\,.
  \end{enumerate}
\end{proposition}
\begin{proof}
  a) follows directly from the additivity of $\chi$ regarding that
    $\chi(C(X,x,a))=1$ for every critical point $x$ for $\varphi$\,,
  b) follows from the calculations in~\cite[\S 5]{Mi}.
\end{proof}

Note that again, up to a set of measure $0$ in $S^{n-1}$, $\varphi(y)=-\langle
a,y \rangle$ is a Morse function for $a\in S^{n-1}$.

Now we are able to define curvature measures:

\begin{definition}
  Let $X\supset U$ be a Borel set.
  \begin{enumerate}[{\upshape a)}]
\item \quad 
$\displaystyle{\sigma(U) = \frac{1}{|S^{n-1}|}\int\limits_{S^{n-1}} \sum_{x\in U}
              \sigma(X,x,a) \,da\;.}$ \\[1mm]
  \item \quad 
    $\displaystyle{\tau(U) = \frac{1}{|S^{n-1}|}\int\limits_{S^{n-1}} \sum_{x\in U}
                  \tau(X,x,a)\,da\;.}$
  \end{enumerate}
\end{definition}

So we get by Proposition~\ref{prop:1.4}:
\begin{proposition}
Let $E^n \supset X$ be compact and belong to a tame class.
  \begin{enumerate}[{\upshape a)}] 
  \item \quad $\sigma(X)  = \chi(X)$ \qquad \text{\Gauss--\Bonnet}\,, \\
  \item \quad $\displaystyle{\tau(X) \geq \sum\limits_{k=0}^n b_k(X)}$ \qquad
    \text{\Chern--\Lashof}\,.
  \end{enumerate}
\end{proposition}

\begin{remark}
  Let $E^n \supset X$ be a manifold $a\in S^{n-1}$ such that $\varphi:y
  \mapsto \langle a,y \rangle$ is a Morse function on $X$ with critical point
  $x$ of index $\lambda$ {\upshape(}compare~\cite[\S\,3]{Mi}{\upshape)}. Then
  $L(X,x,a)$ is homotopically equivalent to $S^{\lambda-1}$ or
  $\emptyset$. Hence
  \begin{align*}
    \sigma(X,x,a) & = 1-\chi(L(X,x,a)) = (-1)^\lambda\,, \\[1mm]
    \tau(X,x,a) & = \biggl|\sum_{k=0}^n b_k(L(X,x,a)) -1 \biggr| =1 \;.
  \end{align*}
So our curvature measures coincide with the usual ones.
\end{remark}

\begin{remark}
  J.\ Fu~\cite{Fu1} studied a large class of ``tame'' sets for which a
  generalized Morse theory exists {\upshape(}independent
  of~\cite{G-M}{\upshape)} and he defined corresponding indices. This work is
  closely related to~\cite{R-Z}.
\end{remark}

\section{PL-spaces}
\label{sec:PL-spaces}

We have seen that we have similar situations for the curvatures $\sigma$ and
$\tau$ respectively. Let us consider this in more generality:

\begin{notation}
  \label{not:2.1}
 Let $\mathcal{T}$ be a tame class. Consider all isotopy classes of pairs
 $(X/Y)$, $X,Y\in \mathcal{T}$, both compact.

 A curvature map $\varrho : \{ (X/Y)\} \to \R$ assigns to each class $(X/Y)$ a
 real number $\varrho(X/Y)$ such that
\begin{enumerate}[{\upshape (i)}]
\item \quad $\varrho\bigl((X/Y)+(X'/Y')\bigr) =
  \varrho(X/Y)+\varrho(X'/Y')$\,,\\[1mm]  
  \mbox{}\quad where $+$ on the left hand side means disjoint union.\\[-1mm]

\item \quad $\varrho\bigl((X/Y)\times(X'/Y')\bigr) =
  \varrho(X/Y)\cdot\varrho(X'/Y')$\,.\\[1mm]
\mbox{} \quad Note that $(X/Y)\times(X'/Y')= (X\times X'\,/\, X\times Y' \cup
X' \times Y)$\,. \\[-1mm]

\item \quad $\varrho(x,x)=0$,\; $\varrho(x,\emptyset)=1$ for a singleton
  $X=x$\,.
\end{enumerate}
\end{notation}

\begin{example}
  \label{exm:2.2}
 Let $b_k(X/Y)$ be the $k^{\text{th}}$ Betti number of $(X/Y)$. The tameness
 of the class $\mathcal{T}$ guarantees that the Betti numbers exist in all
 cases we need.
 \begin{enumerate}[{\upshape a)}]
 \item \quad $\displaystyle{\varrho(X/Y)=\sum_{k=0}^\infty (-1)^kb_k(X/Y)}$\,,\\
 \item \quad $\displaystyle{\varrho(X/Y)=\sum_{k=0}^\infty b_k(X/Y)}$\,,\\
 \item \quad $\displaystyle{\varrho(X/Y)= b_0(X/Y)}$\,.
 \end{enumerate}
This follows from the K\"unneth formula~\cite[XI, \S\,6]{Ma}.
\end{example}

 Now let $E^n\supset X$, $X\in \T$, $x\in X$ and $S$ be the $(n-1)$-sphere in
 $E^n$. According to Definition~\ref{def:1.1} we set, for $a\in S$,
\begin{align*}
  \varrho(X,x,a) &:= \varrho\bigl(C(X,x,a)/L(X,x,a)\bigr)\\
\intertext{and also}
  \varrho(X)&:=\frac{1}{|S|} \int_S \sum_{x\in X}\varrho(X,x,a)\,da\;.
\end{align*}
The tameness guarantees that, up to a set of measure $0$ in S, the sum is
finite. 

More generally, we get a curvature measure $\varrho(X,U)$ assigning to every
open set $U\subset X$ the value
\[
\frac{1}{|S|} \int_S \sum_{x\in U}\varrho(X,x,a)\,da\;.
\]
The cases a{\upshape )} and b{\upshape )} in Example~\ref{exm:2.2} lead to the
curvature measures $\sigma$ and $\tau$ respectively which we considered in
Section~\ref{sec:curvmaesr}.

From now on let $\T= \PL$, the class of piecewise linear spaces.

\begin{notation}
  \label{not:2.3}
  Let $E^n\supset X \in \PL$, $X$ compact. For $0\leq k \leq n$ we denote by
  $X_k$ the $k$-skeleton of $X$.

The curvature measure on $X$ is a Dirac measure, concentrated at the
  vertices of $X$. So
\[
   \varrho(X) = \sum_{x\in X_0}\varrho(x)\;.
\]
Let $F\in X_k$. We denote by $\Aff(F)$ the affine hull of $F$. Let $\Int(F)$
be the set of interior points of $F$. Also, for $x\in \Int(F)$ let
$\Aff(F)^\perp$ be the affine orthogonal complement of $\Aff(F)$ such
that $x\in \Aff(F)^\perp$.

Let $a\in S^{n-1}$. We set
\begin{align*}
  \varrho(X,F,x,a)&:= \varrho(X\cap \Aff(F)^\perp,x,a)\;. \\
 \intertext{This is obviously independent of $x$, so we set}
  \varrho(X,F,a)&:= \varrho(X,F,x,a)\;, 
 \intertext{and correspondingly}
  \varrho(X,F)&:= \bigl\lvert S^{n-k-1}\bigr\rvert^{-1}\!\!\int_{S^{n-k-1}}
  \backIII\varrho(X,F,x,a)\,da\;. 
 \intertext{Finally, we set}
  W_k(F)&:=|F|\,\varrho(X,F)\,, \intertext{where $|F|$ is the $k$-volume of
    $F$, and} W_k(X)&:=\sum_{F\in X_k} W_k(F)\;.
\end{align*}
For $\varrho=\sigma$ one has $W_k(X)=V_k(X)$ (see the Introduction).
\end{notation}

\begin{remark}
 \label{rem:2.4}
 Let $x\in X$, $a\in S^{n-1}$ and $H=\Set{y\in E | \langle y-x,a\rangle
   =0}$\,.

  Assume that the following condition holds:
  \begin{equation}
    \label{eq:condition}
   \text{\emph{$H$ does not contain any face $F$ of $X$ unless $F=\{x\}$}\,.}
\tag{$*$}
  \end{equation}
Then the following pairs of spaces are isotopic:
\begin{enumerate}[{\upshape\ a)}]
\item \quad $\bigl(C(X,x,a)\,/\,L(X,x,a)\bigr)$\\

\item \quad $\bigl(C'(X,x,a)/ L'(X,x,a)\bigr)$\\[1mm]
\mbox{} ~\qquad \quad  $:=\bigl( X \ccap B^n(x,r)\, /\, X \ccap S^{n-1}(x,r) \ccap \Set{ y\in C | \langle
    a,y \rangle\geq 0}\bigr)$\,, \\[1mm]
  \qquad where $r>0$ is so small that $B^n(x,r)\ccap X_0=\emptyset$ or
  $\{x\}$\,.
\end{enumerate}

Note that for $x\in X$ 
condition~\eqref{eq:condition} holds for all $y\in S^{n-1}$ up to a set of
measure $0$\,.
\end{remark}
 
Now let $X_1$ and $X_2$ be two compact $\PL$-spaces in $E^n$, and let $G$ be
the group of all euclidean motions.  For $g\in G$, and $X_1 \cap g X_2$, we
have to look at the vertices and their curvature measures.  Such a vertex
appears if $F_1$ is a $k$-face of $X_1$, $F_2$ an $(n-k)$-face of $F_2$, $x\in
F_1$, $z\in F_2$ and $g(z)=x$. Since the curvature measures are $G$-invariant,
we may consider the following situation:

$F_1$ is a $k$-face of $X_1$, $F_2$ an $(n-k)$-face of $X_2$, 
$\Int(F_1)\cap \Int(F_2)=\{x\}$, say, $x=0$\,. Let $E_i$ be the linear
hull of $F_i$, $i=1,2$, and assume, at first, that $E^n = E_1\perp E_2$. Then
$x$ is a vertex of $E_1\cap X_2$ and $E_2\cap X_1$. In this situation we have
the fundamental

\begin{proposition}
  \label{prop:2.5}
 Let $a\in S^{n-1}$ such that~\eqref{eq:condition} holds for $a$ with respect
 to $X_1\cap X_2$ at $0$. Moreover, let $E_1 \perp E_2$. Let $a=a_1+a_2$,
 $a_i\in E_i$ and $\bar{a}_i= |a_i|^{-1}a_i$.

 {\upshape (}Note that by condition~\eqref{eq:condition} we have $a_i\neq 0$
 for $i=1,2$.{\upshape)}

Then 
\[
  (C/L) \simeq (C_1/L_1) \times (C_2/L_2)\,, 
\]
where
\begin{alignat*}{2}
  C &= C(X_1\times X_2,0,a), & \qquad L&=L(X_1\cap X_2,0,a)\,,\\
  C_i &= C(X_j\cap E_i,0,\bar{a}_i), &\qquad  L_i&=L(X_j\cap E_i,0,\bar{a}_i)\,,
\end{alignat*}
for $i=1,2$ and $j=(i\!\!\mod 2)+1$\,.
\end{proposition}
\begin{proof}
  Inside a ball $B(0,r)$, $r$ sufficiently small, we have $X_i \cap
  B(0,r)=(X_i\cap E_j) \times E_i$. Therefore every $x\in X_1 \cap X_2 \cap
  B(0,r)$ corresponds to a pair $(x_1,x_2)$ with $x_i \in X_i \cap E_j$.

 Now~\eqref{eq:condition} holds also for $X_i\cap E_j = X_i \cap E_i^\perp$
 with respect to $\bar{a}_i$, $i=1, 2$. For if we had a face $\{ 0\} \neq F_i
 \subset \bar{a}_i^\perp \cap B(0,r)$, then $F_i \times \{ 0\}$ would be a
 face in $X_1\cap X_2 \cap B(0,r)\cap a^\perp$.  So we may prove the claim for
 pairs according to type b) in Remark~\ref{rem:2.4}.

Let
\begin{alignat*}{3}
  C' &= X_1\cap X_2 \cap B(0,r), & \quad L'&=X_1\ccap X_2 \ccap S(0,r) & \ccap
  {} & \Set{ y \in X_1\cap X_2 | \langle a, y\rangle \geq 0}\,,\\
  C'_i &= E_i\cap X_j \cap B(0,r/2), & \quad L'_i &=E_i\cap X_j \cap S(0,r/2)
  & \ccap {} &\Set{ y \in E_i\cap X_j | \langle \bar{a}_i, y\rangle \geq 0}\,.
\end{alignat*}
The addition 
\[
  \alpha : E_1 \times E_2 \to E,\quad (x_1,x_2) \mapsto x_1 + x_2
\]
restricts to a homeomorphism $\alpha : C'_1 \times C'_2 \to K \subset C'$,
where $K$ contains a neighbourhood of $0$ in $C'$.
We call  $\Fr(K) := \Set{
     \text{endpoints of half\-lines in } K}$ the \emph{frontier} of $K$\,.
Here
\[
   \alpha\bigl(\,(C'_1 \cap S(0,r/2)) \times (C'_2 \cap
   S(0,r/2))\,\bigr) = \Fr(K)\;.
\]
Now
\[
\alpha\bigl((L'_1 \times C'_2) \cup (L'_2 \times C'_1)\bigr) = 
  M:=\Set{ y \in \Fr(K) |\langle a, y\rangle \geq 0 }\,.
\]
On the other hand, the natural retraction
\[
  \ret: B(0,r)\setminus \{0\}  \to S(0,r)
\]
extends to a homeomorphism $(K/M)\to (C'/L')$\,.
\end{proof}

\begin{remark}
  Proposition~\ref{prop:2.5} remains true if $E_1$ and $E_2$ are not
  orthogonal. To see this, introduce a new scalar product
  $\langle\;,\,\rangle'$ by setting $\langle\;,\,\rangle'=
  \langle\;,\,\rangle$ on $E_i$, $i=1,2$ and $\langle e_1,e_2\rangle'=0$ for
  $e_i \in E_i$. However, then one can  no longer identify $E_i$ and
  $E_j^\perp$ with respect to $\langle\;,\,\rangle$. This will play a role in
  Propositions~\ref{prop:3.2} and~\ref{prop:3.3}.
\end{remark}

\section{The kinematic formula}
 \label{sec:proofkinemat}
\enlargethispage{5mm}
 We keep the situation of Section~\ref{sec:PL-spaces}. So $E=E^n$ is an
 euclidean space, $E = E_1 \perp E_2$, $\dim E_1=k$, $\dim E_2 = n-k$. Let
 $S$, $S_1$, $S_2$ be spheres of dimension $n-1$, $k-1$, $n-k-1$ in $E$,
 $E_1$, $E_2$ respectively.  
 In $S$ we have subspheres $T_i:=S \cap E_i$,  $i=1,2$. 
 Let $S':= S\setminus (T_1 \cup T_2)$.  Then $S'$ is dense in $S$. For $s\in
 S'$ let $s=s_1 + s_2$ with $s_i\in E_i$.

\begin{remark}
  Let $PS^1:=\Set{ (t_1,t_2) \in S^1 | t_1> 0, t_2 >0 }$. Then
\[
f: S' \to PS^1 \times S_1 \times S_2\;,\quad s\mapsto \bigl((\|s_1\|,
\|s_2\|), \bar{s}_1, \bar{s}_2 \bigr)
\]
is a diffeomorphism, where $\bar{s}_i= \|s_i\|^{-1}s_i$.  The inverse is
\[
   f^{-1}: PS^1 \times S_1 \times S_2 \to S' \;, \quad \bigl((t_1, t_2),
   u_1, u_2 \bigr) \mapsto t_1u_1 + t_2u_2\;.
\]
One has for the volume element
\[
   dS' = t_1^{k-1}t_2^{n-k-1}d(t_1,t_2)\, du_1\, du_2\;.
\]
\end{remark}

We come back to the situation of Proposition~\ref{prop:2.5}: $E\supset
X_1,X_2$ are compact $\PL$-spaces.  $F_i$ is a face in $X_i$ and
$E_i=\Aff(F_i)$ for $i=1,2$. Moreover, $\{0\}=\Int(F_1)\cap\Int(F_2)$.

\begin{proposition}
  \label{prop:3.2}
Let $E_1\perp E_2$. Then
\[
  \varrho(X_1\cap X_2,0) = \varrho(X_1\cap E_2,0)\, \varrho(X_2\cap E_1,0) 
  = \varrho(X_1,F_1)\, \varrho(X_2,F_2)\;.
\]
\end{proposition}
\begin{proof}
  \begin{align*}
\MoveEqLeft[3]
  \varrho(X_1 \cap X_2, 0) \\
   = {} & |S|^{-1} \int\limits_S \varrho(X_1 \cap X_2, 0, s)\, ds \\
   = {} & |S|^{-1} \smashoperator{\int\limits_{PS^1 \times S_1 \times S_2}} 
          t_1^{k-1}t_2^{n-k-1} \varrho(X_1 \cap E_2,0,u_2)\,
            \varrho(X_2 \cap E_1,0,u_1)\, d(t_1,t_2)\, du_1\, du_2 \\
   = {} & |S|^{-1} \smashoperator{\int\limits_{PS^1}} t_1^{k-1}t_2^{n-k-1}
           \smashoperator{\int\limits_{S_1 \times S_2}} 
           \varrho(X_1 \cap E_2,0,u_2)\, \varrho(X_2 \cap E_1,0,u_1)
           \, du_1\, du_2\, d(t_1,t_2) \\
   = {} & \varrho(X_1 \cap E_2,0)\, \varrho(X_2 \cap E_1,0)\; |S|^{-1}%
         \smashoperator{\int\limits_{PS^1}} 
           t_1^{k-1}t_2^{n-k-1}\, |S_1|\,|S_2|\, d(t_1,t_2) \\
  = {} & |S|^{-1}\, |S|\, \varrho(X_1 \cap E_2,0)\, \varrho(X_2 \cap  E_1,0)\;.
\qedhere
  \end{align*}
\end{proof}
The preceding proposition is no longer true if $E_1$ and $E_2$ are not
perpendicular (but still complementary), since the measure on $S_i$, which one
gets from projecting $E_j^\perp \to E_i$, is different from the canonical
measure.

So we have to change the measures $du_1$ and $du_2$ by $\mu_1(s_1)\,ds_1$
and $\mu_2(s_2)\,ds_2$ respectively, where $\mu_i$ is a smooth function on
$S_i$ for which
\begin{equation}
  \label{eq:1}
   \int_{S_i}\mu_i(s_i)\, ds_i = 1\,,\quad i=1,2\,. 
\end{equation}
We get as before
\begin{align}
\MoveEqLeft[4]  
\varrho(X_1 \cap X_2,0)  \notag \\
  = {} & |S|^{-1}\smashoperator{\int_S} \varrho(X_1 \cap X_2,0,s)\,ds \notag  \\
  = {} & 
  \begin{multlined}[t]
    |S|^{-1}\smashoperator{\int_{PS^1\times S_1 \times S_2}}
    t_1^{k-1}t_2^{n-k-1}\varrho (X_1 \cap E_2,0,s_2)\,\varrho (X_2 \cap
    E_1,0,s_1){}\cdot{}\qquad\qquad \\[-1.5mm]
  \,d(t_1,t_2)\,\mu_1(s_1)\,ds_1\,\mu_2(s_2)\,ds_2
  \end{multlined}
 \notag  \\[1.5mm]
  = {} &
  \begin{multlined}[t]
    |S|^{-1} \smashoperator{\int_{PS^1}} t_1^{k-1}t_2^{n-k-1}
     \smashoperator{\int_{S_1 \times S_2}} \varrho (X_1 \cap E_2,0,s_2)\,
    \varrho (X_2 \cap E_1,0,s_1){}\cdot{}\qquad\qquad\\[-1.5mm]
    \,\mu_1(s_1)\,ds_1\,\mu_2(s_2)\,ds_2\,d(t_1,t_2) \;.
  \end{multlined}
\label{eq:2}  
\end{align}
Nothing can be said about this. Therefore we take averages. Let $G_i = O(E_i)$
for $i=1,2$.

\begin{proposition}
  \label{prop:3.3}
  \begin{align*}
     \int_{G_1\times G_2}\!\! \varrho(g_2X_1 \cap g_1X_2)\, dg_1\,dg_2 
 & =\varrho(X_1 \cap E_1^\perp,0) \cdot \varrho(X_2 \cap E_2^\perp,0) \\
 & = \varrho(X_1, F_1) \cdot \varrho(X_2, F_2)\;.
 \end{align*}
\end{proposition}
\begin{proof}
 By the preceding formula~\eqref{eq:2} we have 
 \begin{align*}
  {} &   \int_{G_1\times G_2}\!\!\! \varrho(g_2X_1 \cap g_1X_2)\, dg_1\,dg_2 \\
 = {} &
 \begin{multlined}[t]
   |S|^{-1}\!\! \int_{G_1\times G_2} \int_{PS^1}\!\! t^{k-1}t^{n-k-1}\!\! 
     \int_{S_1 \times S_2}\backIII \varrho (g_2 X_1 \cap E_1^\perp,0,s_2)\,
     \varrho (g_1 X_2 \cap E_2^\perp,0,s_1){}\cdot{}\qquad\quad \\[-1.5mm]
    \,\mu_1(s_1)\,ds_1\,\mu_2(s_2)\,ds_2\;d(t_1,t_2)\, dg_1\,dg_2
 \end{multlined}\\[2mm]
 = {} & 
 \begin{multlined}[t]
   |S|^{-1} \int_{PS^1}\!\! t^{k-1}t^{n-k-1}\!\!  \int_{G_1\times G_2}\int_{S_1 \times
     S_2}\backIII \varrho (g_2 X_1 \cap E_1^\perp,0,s_2)\,
   \varrho (g_1 X_2 \cap E_2^\perp,0,s_1){}\cdot{} \qquad\quad \\[-1.5mm]
   \,\mu_1(s_1)\,ds_1\,\mu_2(s_2)\,ds_2\, dg_1\,dg_2\; d(t_1,t_2)
 \end{multlined}\\[2mm]
\intertext{after substituting $s_i\mapsto g_i(s_i)$ and reversing the ordering
  of integrations}
 = {} & 
 \begin{multlined}[t]
   |S|^{-1} \int_{PS^1}\!\! t^{k-1}t^{n-k-1}\!\! \int_{S_1 \times S_2} 
 \int_{G_1\times  G_2}\backIII\!
   \varrho \bigl(g_2 X_1 \cap E_1^\perp,0,g_2(s_2) \bigr)\, \varrho
   \bigl(g_1 X_2 \cap E_2^\perp,0,g_1(s_1)\bigr){}\cdot{}\\[-1.5mm]
   \,\mu_1(g_1(s_1))\,\mu_2(g_2(s_2))\,ds_1\,ds_2\;d(t_1,t_2)
 \end{multlined}\\[2mm]
\intertext{by formula~\eqref{eq:1} and since $\varrho \bigl(g_j X_i \cap
  E_i^\perp,0,g_j(s_j)\bigr) = \varrho (X_i \cap E_i^\perp,0,s_j)$}
 = {} & |S|^{-1} \int_{PS^1}\!\! t^{k-1}t^{n-k-1}\int_{S_1 \times S_2}
 \varrho(X_1 \cap E_1^\perp,0,s_2)\,\varrho(X_2 \cap
   E_2^\perp,0,s_1)\,ds_1\,ds_2\;d(t_1,t_2) \\
 = {} & \varrho (X_1 \cap E_1^\perp,0)\, \varrho (X_2 \cap E_2^\perp,0)  \;.
  \qedhere
 \end{align*}
  \end{proof}

\begin{notation}
  \label{not:3.4}
  Let $E^k$ and $E^{n-k}$ be complementary subspaces of $E^n$. Let
  $e_1,\dots,e_k$ be an orthonormal basis of $E^k$ and $e_{k+1},\dots,e_n$ be
  an orthogonal basis of $E^{n-k}$.
Then we denote
\[
  | E^k : E^{n-k} | := | \det(e_1,\dots,e_n)|\,.
\]
\end{notation}

\begin{remark}
  \label{rem:3.5}
 Let $F_1$ be a $k$-face in $X_1$, $F_2$ an $(n-k)$-face in $F_2$ such that
 $\Aff(F_i)=E_i+a_i$, $0\in E_i$ for $i=1, 2$. Then
\[
   \int_{F_1\,\cap\; F_2 +x \neq \emptyset} \backV dx = |F_1|\,|F_2|\,|E_1 : E_2 |\;.
\]  
\end{remark}

Recall that $G$ is the group of all euclidean motions of $E$ with canonical
measure $dg$. From Proposition~\ref{prop:3.3} and Remark~\ref{rem:3.5} we get

\begin{proposition}
  \label{prop:3.6}
  \begin{align*}
\int_{\atop{G}{x\in F_1\cap g F_2}}
\backIV\backIII
\varrho(X_1 \cap gX_2,x)\,dg & = |F_1|\,|F_2|\, c(n,k)\, \varrho(X_1,F_1)\,
  \varrho(X_2,F_2) \\
    & = c(n,k)\, W_k(F_1)\, W_{n-k}(F_2)\,,\\
\intertext{where}
   c(n,k) &= \!\!\int_{O(E)}\! | E_1 : gE_2| \,dg\;.
  \end{align*}
\end{proposition}

Note that $|E_1 : gE_2|$ only depends on the coset $O(E)\,/\, O(E_1)\times
O(E_2)$.

One may compute $c(n,k)$ directly or as in the proof below.  Now, for the
computation  of $\int_G\varrho (X_1 \cap gX_2)\,dg$ we just gather all
occurrences where a $k$-face of $X_1$ intersects an $(n-k)$-face of $X_2$
and apply Proposition~\ref{prop:3.6}.

\begin{theorem}[Kinematic formula]
 Let $E=E^n$ be an euclidean vector space, $G=O(E)\ltimes E$ the group of all
 euclidean motions endowed with the product of canonical Haar measure and
 Lebesgue measure and let $E\supset X_1,X_2$ be compact $\PL$-spaces. Then
 \[ \int_G \varrho(X_1 \cap gX_2)\, dg = \sum_{k=0}^n \binom{n}{k}^{-1}
  \frac{\omega_k\, \omega_{n-k}}{\omega_n}\, W_k(X_1)\,W_{n-k}(X_2)\;.
 \]
\end{theorem}

\begin{proof}
It remains to compute the universal constants $c_k$.
For this consider the unit cube $Q_k$ of dimension $k$. Then 
\begin{align*}
  \int_G \varrho(X_1 \cap g(r Q_i))\,dg 
  & = \sum^n_{i=n-k} c_i W_i(X_1)\,W_{n-i}(r Q_k) \\
  & = \sum^n_{i=n-k} c_i W_i(X_1)\cdot r^{n-i}\, W_{n-i}(Q_k)\;.
\end{align*}
Dividing by $r^k$ and letting $r\to \infty$ we get
\begin{equation}
  \label{eq:cni}
 c_{n-k}\,W_{n-k}(X_1) = \int\limits_G \varrho(X_1 \cap gE^k)\,dg\;. 
\end{equation}
Next, interchanging $X_1$ and $X_2$, we observe that $c_k = c_{n-k}$ for
$k=0,\dots,n$.

Let $B_k$ be the unit ball in $E^k$, $\omega_k = \vol_k(B_k)$.
We may choose $\varrho = \sigma$.

After approximating $B_n$ by a convex polytope, we may compute
\[
 p(r):=\int\limits_G\varrho\bigl(B_n \cap g(rB_n)\bigr)\,dg 
\]
in two ways. Directly we get
\begin{align*}
p(r) &= \omega_n\, (1+r)^n  = \omega_n \sum_{k=0}^n \binom{n}{k}r^k\,,\\
\intertext{and by~\eqref{eq:cni} we get}
p(r) & = \sum_{k=0}^n c_k\, W_{n-k}(B_n)\,W_k(B_n)\,r^k \\
     & = \sum_{k=0}^n c_k^{-1} \int\limits_G \varrho(B_n \cap gE^k)\, dg 
     \int\limits_G \varrho(B_n \cap gE^{n-k})\, dg\, r^k \\
     & = \sum_{k=0}^n c_k^{-1} \omega_k\,\omega_{n-k}\,r^k\;.
\end{align*}
Comparing homogenous parts we get
\[
   c_k = \binom{n}{k}^{-1} \frac{\omega_k\, \omega_{n-k}}{\omega_n}\;.
\qedhere
\]
\end{proof}

\begin{corollary}[Linear kinematic formula]\mbox{}\par
\label{cor:3.8}
  Let $E^n\supset X$ be a compact $\PL$-space. Then
\[ 
 W_k(X) = \binom{n}{k} \frac{\omega_n}{\omega_k\,\omega_{n-k}}
  \int\limits_G \varrho\bigl(X\cap gE^{n-k}\bigr)\,dg =  
  \smashoperator{\int\limits_{\Graff(n,n-k)}}\varrho(X\cap E)\,dE\;.
\]
\end{corollary}

For $\varrho=\tau$ this holds also for manifolds~\cite{Ba} and even for sets
of positive reach~\cite{R-Z}, but as we mentioned before, very little is
known for the absolute curvature in more general situations.

\end{document}